\documentclass[a4paper,12pt]{article}
\usepackage{bm}
\usepackage{amsmath}
\usepackage{amsthm}
\usepackage{wasysym}
\usepackage{amssymb}
\usepackage{amsfonts}
\usepackage{graphicx}
\usepackage[english]{babel}
\usepackage[utf8]{inputenc}
\usepackage[T1]{fontenc}
\usepackage{mathrsfs}
\usepackage{enumerate}
\usepackage{version}
\usepackage{calc}
\usepackage{subfigure}
\usepackage{authblk}
\usepackage{comment}

\usepackage{mathtools}
\mathtoolsset{showonlyrefs=true}
 \usepackage{pstricks,pst-math,pst-xkey}
 \usepackage{float}
 \usepackage{indentfirst}
\newtheorem{example}{Example}
\newtheorem{definition}{Definition}

\theoremstyle{plain}
\newtheorem{remark}{Remark}

\newtheorem{lemma}{Lemma}
\newtheorem{theorem}{Theorem}

\def\ve{\varepsilon}

\def\d{\partial}

\def\R{\mathbb{R}}
\def\E{\mathbb{E}}

\def\tu{\tau}

\newcommand{\aA}{{\mathbb A}}

\begin{document}

\title{Strong solutions and asymptotic behavior of bidomain equations with random noise.}
\author[2]{Oleksiy Kapustyan \thanks{kapustyanav@gmail.com}}
\author[1]{Oleksandr Misiats\thanks{omisiats@vcu.edu}}
\author[2]{Oleksandr Stanzhytskyi \thanks{ostanzh@gmail.com}}

\affil[1]{Department of Mathematics, Virginia Commonwealth University,
Richmond, VA, 23284, USA}
\affil[2]{Department of Mathematics,
Taras Shevchenko National University of Kyiv, Ukraine}

\date{\today}
\maketitle

\begin{abstract}
In this paper we study the conditions for the existence of strong solutions (both local and global) for stochastic bidomain equations. To this end, we use apriori energy estimates and Serrin-type theorems. We further address the asymptotic behavior of the solutions, which includes the analysis of small stochastic perturbations and large deviations. In a separate section we specify the support of the invariant measure,  whose existence was established in \cite{Mi}.
\end{abstract}



\maketitle

\section{Introduction.} In this paper we study the existence of strong solutions and the asymptotic behavior of nonlocal reaction diffusion-type system of equations, called the {\it bidomain equation}. The bidomain model was first proposed by Tung \cite{Tun}. It is now the generally accepted model of electrical behavior of cardiac tissue \cite{Hen}. The heart muscle consists of so-called {\it excitable cells}, which react to electrical stimulus by depolarizing. In other words, the potential difference across the cell membrane changes. This depolarization motivates us to consider the tissue consisting of two separate domains: the intracellular and the extracellular. The corresponding electric potentials are denoted with $u_i$ and $u_e$, and the difference between these potentials is the transmembrane potential, denoted with $u:=u_i - u_e$. Since the current, which passes from one domain to the other, must pass through the cell membrane, the assumptions on the conductive properties of the cell membrane, combined with the conservation of current principle, yield the system of equations. It consists of a pair of partial differential equations of reaction-diffusion type, coupled with an ordinary differential equation:
\begin{equation}\label{bidomain}
\begin{cases}
\frac{\d u}{\d t} + f(u,w) + A_i u_i = I_i \text{ in } (0, +\infty) \times D,\\
\frac{\d u}{\d t} + f(u,w) - A_e u_e = -I_e \text{ in } (0, +\infty) \times D,\\
\frac{\d w}{\d t} + g(u,w) = 0, \text{ in } (0, +\infty) \times D,\\
\sigma_i \nabla u_i \cdot n = \sigma_e \nabla u_e \cdot n = 0,  \text{ in } (0, +\infty) \times \d D,\\
u(0):=u_i(0) - u_e(0) = u_0, w(0) = w_0 \text{ in } D.
\end{cases}
\end{equation}

Here $D \subset \R^3$ is a smooth bounded set representing myocardium, $A_{i,e}u:=- \nabla \cdot (\sigma_{i,e} \nabla u)$, $\sigma_{i,e}$ are uniformly elliptic conductivity matrices, $f, g: \R \times \R \to \R$ are the functions representing transmembrane ionic currents, and $s_{i,e}$ are the external current sources.

There is a number of models for the nonlinearities $f$ and $g$, which describe the propagation of impulses. The most relevant are the {\it FitzHugh - Nagumo model} \cite{Fit}
\begin{equation}\label{FN}
\begin{cases}
f(u,w) = \eta [u(u-a)(u-1) + w], \\
g(u,w) = bw - cu,
\end{cases}
\end{equation}
the {\it Aliev-Panfilov model} \cite{AliPan}
\begin{equation}\label{AP}
\begin{cases}
f(u,w) = \eta [k u(u-a)(u-1) + wu],  \\
g(u,w) = ku(u-1-a)+w,
\end{cases}
\end{equation}
and the {\it Rogers-McCulloch model} \cite{RogMcc}
\begin{equation}\label{RM}
\begin{cases}
f(u,w) = \eta [b u(u-a)(u-1) + wu],  \\
g(u,w) = -(cu-dw).
\end{cases}
\end{equation}
In these models the coefficients satisfy $0<a<1$ and  $\eta, b,c,d,k,\ve > 0$. Another important example of nonlinear interaction is {\it Allen-Cahn model}
\begin{equation}\label{AC}
\begin{cases}
f(u) = \eta(u^3 - u),\\
g \equiv 0,
\end{cases}
\end{equation}
in which the system \eqref{bidomain} decouples. Following the procedure described, e.g. in \cite{Yves} and references therein, the system \eqref{bidomain} may be written in the form
\begin{equation}\label{operatorBidomain}
\begin{cases}
  \frac{\d u}{\d t} +f(u,w) + \aA u  = I,\\
  \frac{\d y}{\d t} + g(u,w) = 0,\\
  u(0) = u_0, w(0)  = w_0,
  \end{cases}
\end{equation}
where $\aA:= A_i(A_i+A_e)^{-1}A_e$ is the {\it bidomain operator} and $I:=I_i-A_i(A_i+A_e)^{-1}(I_i+I_e)$ is the modified source term. The primary difficulty arising in the analysis of the system \eqref{operatorBidomain} is that the operator $\aA$ is nonlocal. In particulrar, it fails to satisfy the maximum principle, and the comparison theorems do not apply.  Nevertheless, it was shown in \cite{Giga} that this operator generates an analytic semigroup in both $L^p$ and $L^{\infty}$. The paper \cite{Yves} showed that the equation \eqref{operatorBidomain} is locally well-posed in the strong sense and globally well-posed in the weak sense for the nonlinearities $f$ and $g$ described above. The recent result \cite{Mattias1} established the existence of strong periodic solutions.

The equation \eqref{operatorBidomain} is an important model of cardiac defibrillation \cite{KeeSne}, \cite{KunRun}. In this model, $I$ is the extracellular stimulation current in the form $I = \sum_{k=0}^{N} u_k(t) \chi_{\Omega_k}(x)$, where $\Omega_k(x)$ are the characteristic functions of electrodes, and $u_k$ are the corresponding pulses.  During defibrillation various regions of heart tissue are in different, usually random, phases of electrical activity (excited, refractory, partially recovered etc), and the purpose of defibrillation is to give an electric impulse that stimulates the entire heart and returns it to its normal (e.g. stationary) state. Therefore, understanding long time behavior of problems of type \eqref{operatorBidomain} is crucial in order to understand whether such recovery can take place, and how quickly it happens.

In this article we consider the bidomain equations driven by white noise:
\begin{equation}\label{BDES2}
\begin{cases}
  du = [- \aA u - f(u,w) + I] dt + dW, \ t \geq 0,\\
  w_t =  -g(u,w),\\
  u(0) = u_0, w(0)  = w_0,
\end{cases}
\end{equation}
where is an infinite-dimensional  $Q$-Weiner process, defined below. 

To the best of our knowledge, the  results on stochastic bidomain equation of type \eqref{BDES2} are relatively recent, and include \cite{MatStr}, \cite{Kar} and \cite{Mi}. In \cite{Kar} the authors established the existence of both martingale and weak solutions under the monotonicity condition on nonlinearity, which holds, in particular, if we pick FitzHugh-Nagumo model. In \cite{Mi} we established the existence and uniqueness of weak solutions without monotonicity condition. Its worth mentioning that this result is applicable to Rogers-McCullogh and Aliev-Panfilov models, which do not satisfy the monotonicity condition. The existence of stationary solution and the corresponding invariant measures was also obtained in \cite{Mi}. Finally, in \cite{MatStr} the authors obtained the existence of global strong solution for \eqref{BDES2} with FitzHugh-Nagumo nonlinearity.

In the present paper we address the following questions:
\begin{itemize}
\item In Section \ref{sec3} we establish the existence of a local strong solution for a large class of nonlinearities, containing not only FitzHugh-Nagumo and Allen-Cahn, but also Rogers-McCulloch and Aliev-Panfilov nonlinearities;
\item Under slightly more restrictive assumptions on the noise, in Section \ref{sec4} we establish the existence of a global strong solution for the same class of nonlinearities in dimension 2;
\item Finally, in Sections \ref{sec5} and \ref{sec6} we study the asymptotic behavior of the solutions, such as small stochastic perturbations, large deviations principle, and specify the support of the invariant measure.
\end{itemize}

\section{Preliminaries.}\label{sec2}

Throughout the paper, $D \subset \R^3$ is a smooth bounded set, $H = L^2(D)$, 
\begin{equation}\label{defH0}
H_0:=\{v \in H, \int_{D} v dx = 0\},
\end{equation} 
and $V = H^1(D) = W^{1}_2(D)$. The pairing $(u,v)$ stands for the dot product in $H$, and the pairing $\langle u, v \rangle$ stands for the dot product in $V$.

We assume the uniform ellipticity assumption on the conductivities $\sigma_i$ and $\sigma_e$, i.e. there are $\sigma_2>\sigma_1>0$ such that for all $x \in D$ and $\xi \in \R^3$ we have
\begin{equation}\label{elliptic}
\sigma_1|\xi|^2 \leq \langle\sigma_{i,e}\xi, \xi\rangle \leq \sigma_2|\xi|^2
\end{equation}
It was shown in \cite{Mattias1} and  \cite{Yves} that the bidomain operator $\aA$ is well defined on
\begin{equation}\label{domain}
D(\aA):=\{u \in H^2(D)\cap H_0, \nabla u \cdot \nu= 0 \text{ on } \partial D\}.
\end{equation}
Furthermore, following \cite{Yves}, one can define the biliniear form $a(u,v)$ for all $u,v \in V:=H^1(D)$, such that
\[
a(u,v) = (\aA u, v)
\]
if $u$ and $v$ are in $D(\aA)$. This biliniar form implicitly defines the bidomain operator in weak sense. Furthermore, under the assumption \eqref{elliptic} it is shown, e.g. in \cite{Yves}, that the bilinear form $a$ is symmetric, continuous and coercive on $V$, i.e. there are constants $\alpha>0$ and $M>0$ such that
\begin{equation}\label{coercivity1}
\forall u \in V, \ \alpha \|u\|_{V} \leq a(u,u) + \alpha \|u\|_H^2
\end{equation}
and
\begin{equation}\label{coercivity2}
\forall u,v \in V, \ |a(u,v)| \leq M \|u\|_{V} \|v\|_{V}.
\end{equation}
Furthermore, there are $0 < \lambda_0 \leq \lambda_1 \leq  ...\leq \lambda_i \leq ...$ in $\R$ and orthonormal (in $H$) basis of $\{\psi_i, i \geq 1\} \in V$ s.t.
\[
\forall v \in V, a(\psi_i, v) = \lambda_i (\psi_i,v).
\]
It follows, e.g., from Theorem 13, \cite{{Yves}}, that $\psi_i \in D(\aA)$ for all $i \geq 1$, therefore $\aA \psi_i = \lambda_i \psi_i$ in $H$.
We may now define the noise as
\[
W(t,x) := \sum_{i=1}^{\infty} \sqrt{\gamma_i} \psi_i(x) W_i(t),
\]
where $W_i$ are independent standard Wiener processes, and
\begin{equation}\label{def_gamma}
\gamma:= \sum_{i=1}^{\infty} \gamma_i < \infty.
\end{equation}
We next describe the class of the nonlinearities $f$ and $g$. Following \cite{Yves}, we assume that the nonlinearities $f(u,w)$ and $g(u,w)$ are of the form
\begin{equation}\label{forms}
f(u,w) = f_1(u)+f_2(u) w, \ g(u,w) = g_1(u)+g_2 w.
\end{equation}
Here $g_2 \in \R$, and $f_1, f_2$ and $g_1$ are continuous real functions, satisfying the following conditions:
\begin{itemize}
\item{{\bf [C1]}} there exist constants $c_i \geq 0 \ (i=1...6)$ such that for any $u \in \R$
\begin{eqnarray}\label{C1}
|f_1(u)| \leq c_1 + c_2 |u|^{3};\\
\nonumber |f_2(u)| \leq c_3 + c_4 |u|;\\
\nonumber |g_1(u)| \leq c_5+c_6 |u|^{2}.
\end{eqnarray}
\item{{\bf[C2]}} $f_1(u)$, $f_2(u)$ and $g_1(u)$ are locally Lipschitz in $u$;
\item{{\bf [C3]}} there exist constants $a, b$ and $c \geq 0$ such that for any $(u,w) \in \R^2$ we have
\begin{equation}\label{C2}
u f(u,w) + w g(u,w) \geq a u^4 - b(u^2 + w^2) - c.
\end{equation}

\end{itemize}


Let $(\Omega, \mathcal{F}, P)$ be a complete probability space, and
$\mathcal{F}_{t}$ be a right-continuous filtration such that $W(t,x)$ is adapted to $\mathcal{F}_t$, and $W(t)-W(s)$ is independent of $\mathcal{F}_s$ for all  $s<t$.

\begin{definition}\label{defWeek0}
An $\mathcal{F}_t$-adapted random process $(u(t,\cdot), w(t,\cdot)) \in V \times H$ is called a weak solution of \eqref{BDES2}, if  for a.e. $t > 0$, for every $v \in V$ and $y \in H$ we have
\begin{equation}\label{StBD1}
(u(t),v) = (u(0), v) - \int_{0}^{t}[a(u(\tau),v) + (f(u(\tau),w(\tau)), v) - ( I(\tau), v )]d\tau + (v, W(t))
\end{equation}
and
\begin{equation}\label{StBD2}
(w(t),y) = (w(0), y) - \int_{0}^{t}(g(u(\tau),w(\tau)),y)\, d\tau.
\end{equation}

\end{definition}
\begin{definition}\label{defStrong0}
An $\mathcal{F}_t$-adapted random process $(u(t,\cdot), w(t,\cdot)) \in D(\aA) \times H$ is called a strong solution of \eqref{BDES2}, if  for a.e. $t > 0$
\begin{equation}\label{StBD3}
u(t) = u(0) - \int_{0}^{t}[\aA u(\tau) + f(u(\tau),w(\tau)) - I(\tau)]d\tau + W(t),
\end{equation}
and
\begin{equation}\label{StBD4}
w(t)= w(0) - \int_{0}^{t} g(u(\tau),w(\tau))\, d\tau.
\end{equation}
\end{definition}
Let $S(t)$ be the semigroup generated by the operator $-\aA$, i.e.
\[
S(t) u_0(x) := u(t,x),
\]
where $u(t,x)$ solves
\[
\begin{cases}
u_t = - \aA u,\\
u(0,x) = u_0(x).
\end{cases}
\]
It follows from \cite{Giga} that $S(t)$ is a analytic semigroup in $L^p(D)$ for $p > 1$.  Note that by definition of $\lambda_k$ and $\psi_k$ we have
   \begin{equation}\label{eigenfunctions}
   S(t) \psi_k = e^{-\lambda_k t} \psi_k \text{ for } k \geq 1.
   \end{equation}
The semigroup $S(t)$ enables us to define the stochastic convolution
\[
W_A(t,x):= \int_{0}^{t} S(t-\tau) d W(\tau) = \sum_{i=1}^{\infty} \gamma_i \int_{0}^{t} S(t-\tau) \psi_i(x) d W_i(\tu).
\]
Due to Sobolev embedding, $V = H^1(D) \subset L^p(D)$ in $\R^3$ if $2 \leq p \leq 6$ \cite{Bre}.\\

We conclude this section with the following Lemma regarding the stochastic convolution, which was established in \cite{Mi}:

\begin{lemma}\label{lem:1}
  Assume
  \begin{equation}\label{coloring_condition}
  \sum_{k=1}^{\infty} \gamma_k \lambda_k^{1/2}< \infty.
  \end{equation}
  Then\\
  \begin{itemize}
  \item[{[i]}]
   For $T \geq 0$ and almost all $\omega
  \in \Omega$ we have
  \begin{equation}\label{important}
    \sup_{t \in [0,T]} \|W_A\|_{L^p(D)}^p \leq C(T,\omega) < \infty.
    \end{equation}
  \item[{[ii]}]
  For all $t \in [0,T]$ and almost all $\omega
  \in \Omega$ we have $W_A(t,\cdot) \in D(\aA)$.
  \end{itemize}
\end{lemma}

\section{Existence of local strong solutions.}\label{sec3}

In this section, we obtain the existence of a local strong solution for \eqref{BDES2} in the most general case, when the nonlinearities $f$ and $g$ satisfy the conditions {\bf[C1]}-{\bf[C3]}. As follows from Lemma \ref{lem:1}, $W_A(t,x) \in D(\aA)$ provided  \eqref{coloring_condition} holds.  Thus by Lemma 5.13 \cite{DapZab92} we have $W_A(t) = \int_{0}^{t} \aA W_A(s) ds + W(t)$. We proceed with the change of variables $U = u - W_A(t)$. This way, the pair $(U,y)$ solves
\begin{equation}\label{Z}
\begin{cases}
dU = (-\aA U - f(U+ W_A(t),y) + I(t))dt;\\
dw=-g(U+W_A(t),w)dt.
\end{cases}
\end{equation}
If $(U,w)$ is a strong solution of \eqref{Z}, then $U \in D(\aA)$, which, in turn, implies that $u = U + W_A \in D(\aA)$, and hence $(u,w)$ is a strong solution of \eqref{BDES2}. Since $S_A(t)$ is an analytic semigroup, the Theorem 5.14 \cite{DapZab96} for $\alpha \in (0,1/2)$, the process $W_A(\cdot)$ has $\alpha$-Holder continuous trajectories in $H = L^2(D)$. Following \cite{Yves}, introduce
\[
Z:=H \times B, \ B = L^{\infty}(D)
\]
and $\mathcal{A}: D(\mathcal{A}) \subset Z \to Z$, $\mathcal{A} z= (-\aA u,0)$ for $z = (u,w) \in Z$. Here $D(\mathcal{A}) = D(\aA) \times B$, where
\[
D(\aA) =\{u \in H: \sum_{i=1}^{\infty} \lambda_i^2 (u,\Psi_i)^2 < \infty \}.
\]
Next, define
\[
Z^{\alpha} =\{u \in H, \sum_{i \geq 0} \lambda_i^{2 \alpha} (u, \Psi_i)^2 < \infty\} \times B
\]
and
\[
\mathcal{A}^\alpha (u,y):= \left(\sum_{i \geq 0} \lambda_i^{\alpha} (u,\Psi_i) \Psi_i, 0\right).
\]
It follows from \cite{Hen}, that for $0 \leq \alpha \leq \beta$ we have an embedding $Z^{\beta} \subset Z^{\alpha}$, and if $\beta = 1$, then $D(\mathcal{A}) \subset Z^{\beta} \subset Z^{\alpha}$. Hence
\begin{equation}\label{inclusion}
W_A \in D((-\aA)^{\alpha}).
\end{equation}
If $\frac{d}{4} < \alpha < 1$, then $Z^{\alpha} \subset B \times B$. Then by Lemma 16 \cite{Yves} we have
\[
\mathcal{F}(z) := (u,w)  \to (f(u,w),g(u,w))
\]
is a locally Lipschits map from $Z^{\alpha}$ to $Z$, provided $f$ and $g$ are locally Lipschits functions. \\
The main result of this section is the following Theorem:
\begin{theorem}
  Assume $\{\gamma_k, k\geq 1\}$ satisfy \eqref{coloring_condition}. Then for any initial conditions $(u_0,w_0) \in Z^\alpha$ there exists $T = T(\omega)>0$ such that there exists $(u(t,x,\omega), w(t,x,\omega)),$ which is a strong solution of \eqref{BDES2} on the interval $[0,T(\omega))$ in the sense of Definition \ref{defStrong0}.
\end{theorem}
\begin{proof}
Using the local Lipschitz condition, we get
\begin{multline}
\|f(U + W_A(t), w) - f(U_1+W_A(t_1), w_1)\|_{H} \\
\leq L(\|U-U_1\|_{Z^{\alpha}}+\|w-w_1\|_{Z^{\alpha}} + L_1\|W_A(t) - W_A(t_1)\|_{D(-\aA)}),
\end{multline}
where $L_1$ is the constant in the embedding $D(\mathcal{A}) \subset Z$. But
\[
\|W_A(t) - W_A(t_1)\|_{D(-\aA)} = \|W_A(t) - W_A(t_1)\|_{H} + \|\aA (W_A(t) - W_A(t_1))\|_{H}.
\]
Thus it follows from Theorem 5.14 \cite{DapZab92} that $W_A(t)$ is $\alpha$ - Holder, with $\alpha \in (0,1/2)$, i.e.
\[
\|W_A(t) - W_A(t_1)\|_{H} \leq C(\omega) |t-t_1|^{\alpha}
\]
with probability 1. Furthermore, since
\[
\aA W_A = \sum_{k=1}^{\infty} \gamma_k \int_{0}^{t} (-\lambda_k) e^{-\lambda_k(t-s)} \psi_k(x) d W_k(s),
\]
we have
\begin{align*}
& \E \|\aA (W_A(t) - W_A(t_1))\|_{H}^2 = \\
& \E \left\| \sum_{k=1}^{\infty} \gamma_k (-\lambda_k) \left(\int_{0}^{t} (-\lambda_k) e^{-\lambda_k(t-s)} \psi_k(x) d W_k(s) - \int_{0}^{t_1} (-\lambda_k) e^{-\lambda_k(t_1-s)} \psi_k(x) d W_k(s)\right)\right\|_H^2 \leq \\
& \E  \left\| \sum_{k=1}^{\infty} \int_{0}^{t_1} (-\lambda_k)\gamma_k (e^{-\lambda_k(t-s)} - e^{-\lambda_k(t_1-s)})\psi_k(x) d W_k(s) + \int_{t_1}^{t} \gamma_k (-\lambda_k) \psi_k e^{-\lambda_k(t-s)} dW_k(s)\right\|_H^2 \leq \\
&\leq 2 \left(\E \sum_{k=1}^{\infty} \gamma_k \lambda_k  \int_{0}^{t_1} (e^{-\lambda_k(t-s)} - e^{\lambda_k(t_1-s)}) dW_k(s) \right)^2 \\
& + 2 \E  \left(\sum_{k=1}^{\infty} \gamma_k \lambda_k  \int_{t_1}^{t} e^{-\lambda_k(t-s)} dW_k\right)^2 := 2J_1 + 2 J_2.
\end{align*}
Using the properties of the stochastic integrals (\cite{DapZab92}, pp.132-133), we have the following estimate:
\[
J_1 \leq  \sum_{k=1}^{\infty} \gamma_k^2 \lambda_k^2 \int_{0}^{t_1} \left(e^{-\lambda_k(t-s)} - e^{-\lambda_k(t_1-s)}\right)^2 ds = \sum_{k=1}^{\infty} \gamma_k^2 \lambda_k^2 \int_{0}^{t_1} \left(\int_{t_1-s}^{t-s} - \lambda_k e^{-\lambda_k \tau} d\tau\right)^2
\]
\begin{equation}\label{I1}
=  \sum_{k=1}^{\infty} \gamma_k^2 \lambda_k^2 \int_{0}^{t_1} \left(\int_{t_1-s}^{t-s} - \lambda_k \tau e^{-\lambda_k \tau} \frac{1}{\tau} d\tau\right)^2 ds \leq C_0 \sum_{k=1}^{\infty} \gamma_k^2 \lambda_k^2 \int_{0}^{t_1} \left(\int_{t_1-s}^{t-s} \frac{d\tau}{\tau} \right)^2 ds.
\end{equation}
Now, for fixed $\gamma \in (0,1/2)$, we estimate the last term in \eqref{I1} as
\begin{equation}\label{estI1}
\int_{t_1-s}^{t-s} \frac{d\tau}{\tau}  = \int_{t_1-s}^{t-s} \frac{\tau^{\gamma-1} d\tau}{\tau^{\gamma}} \leq  \int_{t_1-s}^{t-s} \frac{\tau^{\gamma-1} d\tau}{(t_1 - s)^\gamma} = \frac{(t-s)^{\gamma} - (t_1-s)^{\gamma}}{\gamma (t_1 - s)^\gamma} \leq C \frac{(t-t_1)^{\gamma}}{(t_1 - s)^\gamma}.
\end{equation}
Combining \eqref{I1} and \eqref{estI1}, we get
\[
J_1 \leq C \sum_{k=1}^{\infty} \gamma_k^2 \lambda_k^2 (t-t_1)^2.
\]
Next, estimating $J_2$ as
\[
J_2 \leq \sum_{k=1}^{\infty} \gamma_k^2 \lambda_k^2 \int_{t_1}^{t} e^{-2 \lambda_k(t-s)} ds \leq C(t-t_1),
\]
altogether we have
\[
\E \|\aA W_A(t)-\aA W_A(t_1) \|^2_H \leq C(t-t_1)^{2\gamma}.
\]
By Proposition 3.15 \cite{DapZab92},
\[
\|\aA W_A(t)- \aA W_A(t_1)\| \leq C(\omega)(t-t_1)^{\alpha}, \, \alpha \in \left(0,\frac{1}{4\gamma}\right).
\]
Hence, by Theorem 20 \cite{Yves}, for any $(u_0,w_0) \in Z^{\alpha}$ there is $T(\omega)$ such that the system \eqref{BDES2} has a strong solution.
\end{proof}

\section{Existence of global strong solution}\label{sec4}

We will next use the regularity theory for critical spaces (Serrin-type results) for parabolic semilinear equations. 
We will need  the following notation:
\begin{itemize}
\item For $\beta \in [0,1]$ the interpolation space $X_\beta = (X_0,X_1)_{\beta} := D(\aA^{\beta}) \times H$, where $X_0 = H \times H$ and $X_1 = D(\aA) \times H$.
\item Time-weighed spaces: if $Y$ is a Banach space, $p > 1$, $a \geq 0$ and $1 \geq \mu > 1/p$, then
\[
u \in L^{p,\mu}((0,a), Y) \text{ if } t^{1-\mu} u \in L^{p}((0,a), Y)
\]
and 
\[
u \in W^1_{p,\mu}((0,a), Y) \text{ if } t^{1-\mu} u \in W^{1,p}((0,a), Y).
\]
\item For $p \geq 2$, define the space $X_{\mu}:=W^{2/p}_{2}(D) \times H$, and 
$E_{\mu}(a) := L^{\infty}((0,a),H) \cap L^2((0,a),V) \times L^p((0,a), H).$ The Sobolev inequality (see, e.g. \cite{Bre}) yields $E_{\mu}(a) \subset L^p([0,a],X_{\mu})$ if $D \subset \R^2$. 
\end{itemize}
Let $A$ be a sectorial operator $A: X_1 \to X_0$, and $F: X_\beta \to X_0$, where $\beta \in (0,1)$ is chosen such that $X_{\beta} \subset V \times H.$ Consider the following evolution equation:
\begin{equation}\label{evolution_equation}
\begin{cases}
v_t + A v = F(v+\varphi(t,x)),\\
v(0) = v_0 \in X_{\beta}.
\end{cases}
\end{equation}

Assume
\begin{itemize}
    \item[{[i]}] For every $a \geq 0$, $\varphi \in L^{2p,\mu}([0,a), X_{\beta})$.
    \item[{[ii]}] There exists $C>0$ such that
\begin{equation}\label{bound}
\|F(\cdot)\|_{X_0} \leq C (1+\|\cdot\|_{X_\beta}^3).
\end{equation}
\end{itemize}
Then the following Theorem holds:
\begin{theorem}\label{Pru}
Assume the conditions [i] and [ii] hold, and the solution  of \eqref{evolution_equation} is defined on the maximal interval of existence $[0,T)$. Then $u \in L^p((0,a), X_{\mu})$ for each $0 \leq a<T$, and if $T < \infty$, then $u \notin L^{p}([0,T), X_{\mu})$. 
\end{theorem}
\begin{proof}
The proof follows the idea from \cite{Pru}, Corollary 2.3. For 
\begin{equation}\label{cond}
\mu - \frac{1}{p} > 2 \beta -1,
\end{equation} 
let 
\[
\alpha:= \frac{\beta - (\mu-1/p)}{1-(\mu-1/p)}.
\]
Assume $T<\infty$. Using the condition [ii], for any $a<T$ we have
\begin{equation}\label{bound1}
\|F(v+\varphi)\|_{L^{p,\mu}((0,a), X_{\beta})} \leq C_1(1+ \|v + \varphi\|^3_{L^{2p,\tau}((0,a),X_{\beta})}) \leq 4C_1 \|v\|^3_{L^{2p,\tau}((0,a),X_{\beta})} + C_2,
\end{equation}
where
\[
C_2:=C_1 + 4C_1 \|\varphi\|^3_{L^{2p,\tau}((0,T),X_{\beta})} < \infty.
\]
Next, by interpolation inequality we have
\begin{equation}\label{bound2}
\|v\|^3_{L^{2p,\tau}((0,a),X_{\beta})} \leq C_3 \|v\|_{L^p((0,a), X_{\mu})}^{3(1-\alpha)} \|v\|_{E_{\mu}(0,a)}^{3\alpha}.
\end{equation}
Finally, let $M$ be the constant of maximal regularity for the interval $[0,T)$. Combining \eqref{bound1} and \eqref{bound2}, we have
\[
\|v\|_{E_{\mu}(0,a)} \leq M(\|v(0)\|_{X_{\mu}} + C_2 + C_4  \|v\|_{L^p((0,a), X_{\mu})}^{3(1-\alpha)} \|v\|_{E_{\mu}(0,a)}^{3\alpha}).
\]
For the subcritical choice of parameters $\mu$, $\beta$ and $p$, given by \eqref{cond} such that $3 \alpha \leq 1$, we have $\|v\|_{E_\mu(0,a)}$ is bounded uniformly in $a$ for all $a < T$. This means that the solution can be continued beyond $T$, which is a contradiction with the fact that $[0,T]$ is the interval of maximal existence.

\end{proof}

The main result of this section is the following Theorem:
\begin{theorem}
  Assume 
  \[
  \sum_{k=1}^{\infty} \gamma_k^2 \lambda_k^2 < \infty,
  \]
and the equation \eqref{BDES2}
is considered in dimension 2. Then for any initial conditions $(u_0,w_0) \in X_{\beta}$ there exists a strong solution $(u(t,x,\omega), w(t,x,\omega))$ of \eqref{BDES2} in the sense of Definition \ref{defStrong0}, which is defined for all $t \geq 0$.
\end{theorem}
\begin{proof}
Let $t \in [0,T)$ be the maximal interval of existence of $(u(t,x,\omega), w(t,x,\omega))$. We argue by contradiction, and assume $T< \infty$. Using generalized Ito's formula \cite{Kry}, we have the following energy bounds:
\begin{align*}
\begin{cases}
 \|u(t)\|^2 = \|u_0\|^2 +&2 \int_{0}^{t} \langle -\aA u, u \rangle d \tau - 2 \int_{0}^{t} (f,u) d\tau + 2 \int_{0}^{t} (I,u) d \tau \\
 & + \gamma t + 2 \int_{0}^{t}(u(\tau), dW(\tau)),\\
\|w(t)\|^2 = \|w_0\|^2 -&2 \int_{0}^{t}(g,w) d \tau.
\end{cases}
\end{align*}
 It follows from \eqref{coercivity1} that
\begin{equation}\label{estA}
 \langle -\aA u, u \rangle \leq - \alpha \|u\|_V^2 + \alpha \|u\|_H^2.
\end{equation}
Combining \eqref{estA} and the condition \eqref{C2}, we get

\begin{multline}\label{energybound}
 \|u(t)\|^2 + \|w(t)\|^2 \leq  \|u_0\|^2_H + \|w_0\|_H^2+  \\
 + 2 \int_{0}^t\left[-\alpha \|u(s)\|_V^2 +\alpha \|u(s)\|_H^2 - 2 a \int_D u^4(x,s) dx + b (\|u(s)\|_H^2 + \|w(s)\|_H^2) + c|D|\right] ds \\
 +  2 \int_{t_0}^{t}(I,u) d \tau + \gamma t + 2 \int_{0}^{t}(u(\tau), dW(\tau)).
\end{multline}

Using Lemma 7.2 \cite{DapZab92}, we have
\[
\left(\E \sup_{s \in [0,t]} \left| \int_0^s (u(\sigma), d W(\sigma)) \right|\right)^2 \leq \E \left(\sup_{s \in [0,t]} \left|\int_0^s (u(\sigma), d W(\sigma))\right|\right)^2 =  \E \sup_{s \in [0,t]} \left|\int_0^s (u(\sigma), d W(\sigma)) \right|^2 \leq
\]
\[
\leq C \int_{0}^t \E \|u(s)\|^2_H ds \leq C \int_{0}^t \E \sup_{\sigma \in [0,s]}(\|u(\sigma)\|^2_H + \|w(\sigma)\|^2_H) ds.
\]
Define
\[
\psi(s):= \E \sup_{\sigma \in [0,s]} (\|u(\sigma)\|_H^2 + \|w(\sigma)\|_H^2),
\]
we have
\[
\E \sup_{s \in [0,t]} \left|\int_{0}^s u(\sigma) dW(\sigma)\right| \leq  C \sqrt{\int_0^t \varphi(s) ds} \leq c_1 + c_2 \int_0^t \varphi(s) ds.
\]
It follows from \eqref{energybound} that for $t \in [0,T]$
\[
\psi(t) + 2 \alpha \E  \int_0^t \|u\|_V^2 ds + 2a \E \int_0^t \int_D u^4 dx ds \leq c_1 + c_2 \int_0^t \psi(s) ds.
\]
Hence, by Gronwalls inequality,  
\[
\psi(t) \leq c_1 e^{c_2 t}, t \in [0,T],
\]
as well 
\begin{equation}\label{L4estimate}
\E \int_0^t \|u\|_V^2 ds + \E \int_0^t \int_D u^4 dx ds  < \infty.
\end{equation}
Consequently, for almost all $\omega \in \Omega$, 
\begin{equation}\label{2}
u \in L^{\infty}((0,T),H) \cap L^2((0,T),V)
\end{equation}
and
\[
w  \in L^{\infty}((0,T),H).
\]
Introduce $z:=u-W$. Then $z$ satisfies
\begin{equation}
\begin{cases}
  dz = [- \aA z - f(z+W,w) + I] dt, \ t \geq 0,\\
  dw =  -g(z+W,w) dt.
\end{cases}
\end{equation}
Since 
\[
W(t,x) = \sum_{k=1}^{\infty} \gamma_k \Psi_k(x) W_k(t),
\]
we have
\[
\aA W(t,x) = \sum_{k=1}^{\infty} \gamma_k \aA \Psi_k(x) W_k(t),
\]
hence
\[
W \in D(\aA) \iff \sum_{k=1}^{\infty} (\gamma_k \lambda_k)^2 < \infty.
\]
Thus $W$ belongs to $V$, and for a.e. $\omega \in \Omega$ we have
\[
\|W\|_{L^2((0,T),V)} < \infty.
\]
Furthermore, by Theorem 2.5 \cite{chow}, we have
\[
\E \sup_{0 \leq t \leq T} \|W(t)\|^2_{H} \leq 4 \E \|W(T)\|_H^2 < \infty,
\]
thus for a.e. $\omega \in \Omega$ we have 
\begin{equation}\label{3}
W \in L^{\infty}((0,T),H) \cap L^2((0,T),V). 
\end{equation}
Combining \eqref{2} and \eqref{3} we get 
that 
\[
z = u - W  \in L^{\infty}((0,T),H) \cap L^2((0,T),V).
\]
We claim that
\begin{equation}\label{w in L^p}
w \in L^p([0,T], H).
\end{equation}
To this end,  we represent the solution of 
\[
w_t = -g_2 w   - g_1(u)
\]
as
\[
w = \frac{1}{e^{g_2 t}} \left(- \int_0^t e^{g_2 s} g_1(u(s)) ds + w_0 \right).
\]
Hence,
\begin{align}
&  \int_0^t \|w(s)\|_{H}^p ds \leq C_1 + C_2 \int_0^t \left\| \int_0^s e^{g_2 \tau} g_1(u(\tau)) d \tau \right\|_{H}^p ds \leq \\
\nonumber &  \leq C_1 + C_2  \int_0^t \left( \int_0^t \| g_1(u(\tau))\| _{H} d \tau \right)^p ds = C_1 + C_2 t  \left( \int_0^t \| g_1(u(\tau))\| _{H} d \tau \right)^p = \\
\nonumber & = C_1 + C_2 t \left( \int_0^t \left(\int_{D} g_1^2 (u(\tau)) dx \right)^{1/2} d \tau \right)^p \leq  C_1 + C_2 t \left( \int_0^t \left(\int_{D} g_1^2 (u(\tau)) dx \right)^{1/2} d \tau \right)^p \leq \\
\nonumber &  \leq C_1 + C_2 t \left( \int_0^t \left(\int_{D} u^4(\tau)) dx \right)^{1/2} d \tau \right)^p < \infty
\end{align} 
a.s., where we used \eqref{L4estimate}. We are now in position to use the embedding for $d=2$:
\[
L^{\infty}((0,T),H) \cap L^2((0,T),V) \subset L^p((0,T),W^{2/p}_2(D)).
\]
This way for $v = (u,w)$ we have
\[
v \in L^p((0,T),W_2^{2/p}(D) \times L^q(D)) = L^p((0,T),X_{\mu})
\]
It remains to verify that $\varphi(t,x):= [-W(t,x),0]$ satisfies the condition [i]. Note that by Theorem 2.3, ch 6 \cite{chow}, we have
\[
\E \sup_{0\leq \tau \leq T}\|W(\tau,x)\|_{X_\beta}  \leq \left(\frac{2p}{2p-1}\right)^{2p}\|W(T,x)\|_{X_\beta} < \infty,
\]
since $W(t,x) \in D(\aA)$. Therefore for a.e. $\omega$
\[
W(t,x) \in L^{\infty}((0,T), X_{\beta}) \subset  L^{2p, \mu}((0,T), X_{\beta}),
\]
thus [i] holds. Hence, by Theorem \ref{Pru}, $T$ must be equal to $+\infty$, leading to contradiction. Therefore, we have global existence of the solution.
\end{proof}

\section{Support of the invariant measure}\label{sec5}

The long time behavior of solutions of stochastic evolution equations in some Hilbert space $H$ is a question of a separate interest. It is often addressed in terms of establishing existence and, in certain cases,  uniqueness of invariant measures, which, in turn, is the crucial item in establishing the ergodic behavior of the underlying physical systems. For the reader's convenience, we briefly remind the concept of invariant measure. Suppose $v_0 \in H$, and $B$ is a Borel subset of $H$. For $t \geq 0$ define the probability transition semigroup
\[
P_t(v_0,B):=\mathbb{P}(v(t,v_0) \in B).
\]
Furthemore, for any Borel measurable bounded function $\varphi \in M_b(H)$, define the Markov semigroup
\[
P_t \varphi(v_0) = \E \varphi(v(t,v_0) = \int_{H} \varphi(v) P_t(v_0,dv).
\]
\begin{definition}
Let the set $Pr(H)$ be the set of probability measures on $H$. An element $\mu \in Pr(H)$ is called an invariant measure for the Markov semigroup $P_t$ if
\[
\int_{H} \varphi(v_0) d \mu(v_0) = \int_{H}P_t \varphi(v_0) d \mu(v_0).
\]
\end{definition}

The works \cite{DapZab92, DapZab96} provide conditions for the existence and uniqueness of invariant measures in for general evolution equation of reaction-diffusion type. Their approach is based on the results of  Krylov and Bogoliubov \cite{KryBog} on the tightness of a family of measures. The key ingredients of this approach include the Feller property and stochastic continuity of the Markov semigroup $P_t$, and the existence of least one solution which is globally bounded in certain probability sense \cite{MisStaYip, MisStaYip2, MisStaYip3}. 
The existence of invariant measures for bidomain equation \eqref{BDES2} using the aforementioned procedure was established in \cite{Mi}. In particular, we showed that the equation \eqref{BDES2} has a stationary solution, which, in turn, implies the existence of invariant measure, defined on the functional space $(u,w)  \in H \times H$. The sufficient conditions for the uniqueness of this invariant measure were also derived in \cite{Mi}. 

A different approach to establishing the existence of invariant measures  and asymptotic behaviour of stochastic evolution equations is based on the coupling method (see e.g. \cite{Mue93}, \cite{Hairer} and references therein).  The study of the asymptotic behavior based on the exponential dichotomy of the differential operator was carried on in \cite{dichotomy}.

The main goal of this section is to show, that the invariant measure, established in \cite{Mi}, Theorem 5.1, is supported on the functional class $(u,w)$ with $u$ having regularity of at least $V$. Recall that for $v = (u,w)$ we denote 
\[
\|v\|_{\tilde{V}}^2: = \|u\|_{V}^2 + \|w\|^2_{H}
\]
and
\[
\|v\|_{\tilde{H}}^2: = \|u\|_{H}^2 + \|w\|^2_{H}.
\]
The main result of this section is the following theorem:

\begin{theorem}
Assume the conditions of Theorem 5.1 \cite{Mi} holds. 
Let $\mu \in Pr({\tilde{H}})$ be the invariant measure for \eqref{BDES2}. Then 
\[
\int_{\tilde{H}} \|v\|_{\tilde{V}}^2 d \mu(v) < \infty.
\]
\end{theorem}
In other words, any invariant measure is supported in $\tilde{V} = V \times H$.
\begin{proof}
The idea of the proof was inspired by the work \cite{Vikol}, in which the authors derived a similar result for the invariant measure emerging in stochastic primitive equations. As it was shown in \cite{Mi} (see proof of Theorem 5.1),  if $v = (u,w)$ is a solution of \eqref{BDES2}, then
\[
\int_{0}^{T} \|v\|_{\tilde{V}}^2 dt =  \int_{0}^{T} (\|u\|_{V}^2 + \|w\|^2_{H}) dt < K_1 \|v_0\|_{\tilde{H}}^2 + K_2 T
\]
for some $K_1$ and $K_2$ independent of $T$.  Denote $f_R(v): =\|v\|_{\tilde{V}}^2 \land  R$. By definition of invariant measure
\begin{equation}\label{inv1}
\int_{\tilde{H}} f(v_0) d \mu(v_0) = \int_{\tilde{H}} \int_{\tilde{H}} \frac{1}{T} \int_{0}^T P_t(v_0,dv) f(v) dt d\mu(v_0)
\end{equation}
for any $T>0$. Now, for any $\rho \geq 1$ and any $v_0 \in B_{\tilde{H}}(\rho)$, the ball of radius $\rho$ about the origin in ${\tilde{H}}$, we have 
\begin{equation}\label{inv2}
\left|\frac{1}{T} \int_{0}^{T} \int_{\tilde{H}} P_t(v_0,dv) f_R(v) dt d\mu(v) dt \right| = \left|\frac{1}{T} \int_{0}^{T} \E f_R(v(t,v_0)) dt \right| \leq K_2 + \frac{K_1 \|v_0\|_{\tilde{H}}^2}{T} \leq K_3(1 + \frac{\rho^2}{T}).
\end{equation}
Combining \eqref{inv1} and \eqref{inv2} we have

\begin{multline}\label{bnd}
 \int_{\tilde{H}} f_R(v_0)  d\mu(v_0) \leq \int_{B_{\tilde{H}}(\rho)} \frac{1}{T} \left|\int_{0}^{T}  \int_{\tilde{H}} P_t(v_0,dv) f_R(v) d \mu(v) dt \right| d\mu(v_0) +\\
 \int_{{\tilde{H}} \setminus B_{\tilde{H}}(\rho)}  \frac{1}{T} \left|\int_{0}^{T}  \int_{\tilde{H}} P_t(v_0,dv) f_R(v) d \mu(v) dt \right| d\mu(v_0)  \leq K_3(1 + \frac{\rho^2}{T}) + R\mu({\tilde{H}}\setminus B_{\tilde{H}}(\rho)). \end{multline}
First, for sufficiently large $\rho$, depending on
$R$, we have
\[
R\mu({\tilde{H}}\setminus B_{\tilde{H}}(\rho)) \leq 1.
\]
Then for such $\rho$ we choose $T$ large enough so that
\[
\frac{\rho^2}{T} \leq 1.
\]
Altogether, \eqref{bnd} yields an estimate, which is independent of $R$:

\begin{equation}\label{bndR}
\int_{\tilde{H}} f_R(v_0) d \mu(v_0) \leq 2 K_3 +1.
\end{equation}
Applying the Monotone Convergence Theorem to \eqref{bndR} we get
\[
\int_{\tilde{H}}\|v_0\|^2_{\tilde{V}} d \mu(v_0) < \infty,
\]
which completes the proof of the theorem.
\end{proof}

\section{Small Perturbations and Large Deviations}\label{sec6}

\subsection{Small perturbations}
In certain physical problems it is natural to assume that the random perturbations are small compared to the deterministic component of the dynamics (see \cite{VenFre} and references therein).  In this section, we consider such perturbations for bidomain equation, which leads to the following equation for any $\ve>0$: 
\begin{equation}\label{BDESeps}
\begin{cases}
  du_\ve  = [- \aA u_\ve - f(u_\ve ,w_\ve ) + I] dt + \ve dW, \ t \geq 0,\\
  d w_\ve  =  -g(u_\ve ,w_\ve ) dt,\\
  u_\ve (0) = u_0, w_\ve (0)  = w_0.
\end{cases}
\end{equation}
Along with \eqref{BDESeps}, we consider the limiting equation
\begin{equation}\label{BDESlim}
\begin{cases}
  du  = [- \aA u - f(u ,w ) + I] dt, \ t \geq 0,\\
  dw  =  -g(u ,w) dt,\\
  u(0) = u_0, w(0)  = w_0.
\end{cases}
\end{equation}

Let $y:=(u,w)^{T} \in \R^2.$ Define $\mathcal{F}$ by
\[
  \mathcal{F}(y) = \mathcal{F}(u,w) := \left(
                       \begin{array}{c}
                         -f(u,w) + I \\
                         -g(u,w) \\
                       \end{array}
                     \right).
  \]
We assume that there exists constants $c_1 \in \R$ and $c_2 \in \R$ such that for all $(u_1,w_1), (u_2,w_2) \in \R^2$ the function $\mathcal{F}$ satisfies the {\em monotonicity condition}
\begin{equation}\label{monot}
     (\mathcal{F}(u_1,w_1) - \mathcal{F}(u_2,w_2))\cdot ((u_1,w_1) - (u_2, w_2))  \leq -c_1 (u_1 - u_2)^2 - c_2 (w_1 - w_2)^2.
\end{equation}

\begin{theorem}
Assume $f$ and $g$ satisfy the conditions {\bf [C1]}-{\bf [C3]}, and the monotonicity condition \eqref{monot}. Then there is $C = C(T)$ such that
\[
\E \sup_{t \in [0,T]}(\|u_\ve(t) - u(t)\|^2 + \|w_\ve(t) - w(t)\|^2)  \leq  \ve C(T).
\]
\end{theorem}
\begin{proof}
Introduce
$v_\ve := u_\ve - u$ and $z_\ve := w_\ve - w.$
Then the pair $(v_\ve, z_\ve)$ solves

\begin{equation}\label{BDEdiff}
\begin{cases}
  dv_\ve  = [- \aA u_\ve - (f(u_\ve, w_\ve )-f(u,w))] dt + \ve dW, \ t \geq 0,\\
  d z_\ve  =  -(g(u_\ve, w_\ve)-g(u,w)) dt,\\
  v_\ve (0) = 0, z_\ve (0)  = 0.
\end{cases}
\end{equation}
Using Ito's formula (see, e.g.\cite{Kry}), we have
\begin{equation*}
\|v_\ve(t)\|^2 = 2 \int_0^t \langle -\aA v_\ve, v_\ve \rangle ds - 2 \int_0^t (f(u_\ve, w_\ve )-f(u,w), u_\ve - u) ds + \ve^2 \gamma t + 2 \ve \int_0^t (v_\ve, d W_s)
\end{equation*}
 as well as
\begin{equation*}
\|z_{\ve}\|^2 = -2  \int_0^t (g(u_\ve, w_\ve)-g(u,w), w_\ve- w) ds.
\end{equation*}
Altogether, using the monotonicity condition \eqref{monot} and the estimate \eqref{estA},
\begin{multline*}
\|v_\ve(t)\|^2 + \|z_{\ve}(t)\|^2  \leq 2 \int_0^t (-\alpha \|v_\ve\|_V^2 + \alpha \|v_\ve\|^2 ) ds + C_1  \int_0^t ( \|v_\ve\|^2 + \|z_\ve\|^2) ds \\
+ \ve^2 \gamma t + 2 \ve \int_0^t (v_\ve, d W_s).
\end{multline*}
Hence
\begin{multline*}
\E \sup_{t \in [0,T]} (\|v_\ve(t)\|^2 + \|z_{\ve}(t)\|^2)  \leq \sigma \int_0^T ( \|v_\ve\|^2 + \|z_\ve\|^2) ds
+ \ve^2 \gamma T + 2 \ve \E \sup_{t \in [0,T]} \left|\int_0^t (v_\ve, d W_s)\right|,
\end{multline*}
for some $\sigma > 0$. 
Here 
\[
\int_0^t (v_\ve, d W_s) = \int_0^t (v_\ve, \sum_{j} \gamma_i \Psi_i(x) d W_i(s))
\]
is a martingale, which can further be estimated as follows:
\begin{align}\label{est_stochast}
& \E  \sup_{t \in [0,T]} \left|\int_0^t (v_\ve, d W_s)\right| \leq \sqrt{\E \sup_{t \in [0,T]} \left|\int_0^t (v_\ve, d W_s)\right|^2}   \\
\nonumber & =  \sqrt{\E \sup_{t \in [0,T]} \left|\int_0^t \left(v_\ve, \sum_{i} \gamma_i \Psi_i d W_i(s)\right)\right|^2}  = 
   \sqrt{\E \sup_{t \in [0,T]} \left| \sum_{i} \gamma_i^2 \int_0^t (v_\ve, \Psi_i) d W_i(s)\right|^2}   \\
 \nonumber  &  \leq C \sqrt{\sum_{i} \gamma_i^2  \int_0^T \E(v_\ve, \Psi_i)^2 ds }  \leq C \sqrt{ \int_0^T \E \|v_\ve(s)\|^2 ds } \leq C +C  \int_0^T  \E \|v_\ve(s) \|^2 ds,  
\end{align}
 where we used the elementary inequality $\sqrt{a}\leq a+1.$ This way, for some $\sigma_1>0$, we have
 \begin{multline*}
\E   \sup_{t \in [0,T]} (\|v_\ve(t)\|^2 + \|z_\ve(t)\|^2) \leq \sigma_1  \int_0^T \sup_{s \in [0,t]} (\|v_\ve(s)\|^2 + \|z_\ve(s)\|^2) dt + \ve^2 \gamma T + 2 \ve C.
 \end{multline*}
 It follows from Gronwall inequality that
 \[
 \E \sup_{t \in [0,T]}(\|u_\ve(t) - u(t)\|^2 +\|w_\ve(t) - w(t)\|^2) \leq C(T) \ve.
 \]
 \end{proof}

\subsection{Large deviations} 
The influence of small random perturbations on large time intervals is a question of special interest. Generally speaking, in this case small perturbations have a significant influence on the macroscopic behavior of the underlying physical system. In order to study this influence, we have to estimate the probabilities of unlikely events. In other words, we need to study the asymptotic behavior of large deviations for random processes. 

Let $C_p$ be the
constant in Poincare's inequality, i.e.,
\begin{equation}\label{Poincare}
\|u\|_{H}^2 \leq C_{p} \|\nabla u\|_{H}^2, \quad u \in V \cap H_0.
\end{equation}
with $H_0$ defined in \eqref{defH0}. The main result of this subsection is the following Theorem:

\begin{theorem}\label{Thm: Large Deviations}
Let the conditions [C1], [C2, [C3] and the monotonicity condition \eqref{monot} holds. Assume, in addition, that the constants $c_1$ and $c_2$, appearing in \eqref{monot}, satisfy $c_2 \geq 0$ and 
\begin{equation}\label{coefficient condition}
c_1 \geq -\frac{\alpha}{C_p}
\end{equation}
with $\alpha$ and $C_p$ given by \eqref{estA} and \eqref{Poincare} respectively. Then  
\begin{equation}\label{eqn:10}
P\{ \sup_{0 \leq t \leq T} \{\|u_\ve(t) - u(t)\|^2 +  \|w_\ve(t) - w(t)\|^2\} \geq r^2\} \leq 3 \exp\left[ -\frac{r^2}{4 \gamma \ve^2 T}\right],
\end{equation}
with $\gamma$ given in \eqref{def_gamma}.
\end{theorem} 
 \begin{proof} We first note that, using \eqref{estA}, \eqref{monot}, \eqref{Poincare} and \eqref{coefficient condition}:
\begin{align}\label{negative}
& \langle -\aA v_\ve, v_\ve \rangle - (f(u_\ve,w_\ve) - f(u,w), u_\ve - u) - (g(u_\ve, w_\ve) - g(u,w), w_\ve - w) \\
\nonumber & \leq -\alpha \|v_\ve\|_V^2 + \alpha \|v_\ve\|_H^2 - c_1  \|v_\ve\|_H^2 - c_2 \|z_\ve\|_H^2  = -\alpha  \|\nabla v_\ve\|^2_H  - c_1  \|v_\ve\|_H^2 - c_2 \|z_\ve\|_H^2 \leq \\
\nonumber & -\frac{\alpha}{C_p} \|v_\ve\|_H^2  - c_1  \|v_\ve\|_H^2 - c_2 \|z_\ve\|_H^2 \leq 0.
\end{align}
Following \cite{chow}, introduce the functional
 \[
 \Phi_{\lambda} (v,z)  = (1 + \lambda(\|v\|^2 + \|z^2\|^2)^{1/2}, \ v,z \in H.
 \]
Using Ito's formula, we have
\begin{align} \label{def:Phi}
 &   \Phi_{\lambda} (v_\ve, z_\ve)  = 1 + \lambda \int_0^t \Phi_{\lambda}^{-1} (v_\ve, z_\ve) [ \langle -\aA v_\ve, v_\ve \rangle \\
\nonumber &  -(f(u_\ve,w_\ve) - f(u,w), u_\ve - u) - (g(u_\ve, w_\ve) - g(u,w), w_\ve - w)  ] ds + \\
\nonumber & + \lambda \ve \int_0^t \Phi_{\lambda}^{-1} (v_\ve, y_\ve)(v_\ve, d W_s) + \frac{\ve^2}{2} \int_0^t \left\{\lambda  \Phi_{\lambda}^{-1} (v_\ve, z_\ve) \gamma - \lambda^2   \Phi_{\lambda}^{-3} (v_\ve, z_\ve)  (Qv_\ve, v_\ve) \right\} ds.
 \end{align}
 Here $Q$ is a trace class covariance operator, such that $Q e_k = \lambda_k e_k$, $Q$ is given in the definition of $W_t$. 
 It follows from \eqref{def:Phi} and \eqref{negative} that
 \begin{multline}\label{eqn:3}
 \Phi_\lambda (v_\ve, z_\ve) \leq 1 + \eta_t^{\lambda} + \\
  \frac{\ve^2}{2} \int_0^t\left[\lambda \Phi^{-1}_\lambda (v_\ve, z_\ve) \gamma + \lambda^2 (Qv_\ve, v_\ve) (\Phi^{-2}_\lambda (v_\ve, z_\ve) - \Phi^{-3}_\lambda (v_\ve, z_\ve))\right] ds,
 \end{multline}
 where $\gamma = Tr(Q)$ and
\begin{equation}\label{def:eta}
\eta_t^{\lambda}  = \lambda \ve \int_0^t \Phi_\lambda^{-1} (v_\ve, z_\ve)(v_\ve, d W_s) - \frac{1}{2} \lambda^2 \ve^2 \int_0^t \Phi_\lambda^{-2} (v_\ve, z_\ve) (Q v_\ve, v_\ve) ds.
\end{equation}
Note that $\forall \lambda>0$ $\forall \alpha > 0$, with $0 < \Phi_\lambda^{-\alpha} \leq 1$, we have 
\[
 \Phi_\lambda^{-2}(v,z) [1 -  \Phi_\lambda^{-1}(v,z)] \leq 1
\]  
and
\[
(Q v, v) \leq \gamma \|v\|^2. 
\]
 Therefore,
 \[
  \Phi_\lambda^{-2}(v,z)(Qv, v) \leq \frac{\gamma}{\lambda}  \frac{\|v\|^2}{\frac{1}{\lambda} + \|v\|^2 +\|z\|^2},
 \]
 which for $t \leq T$ implies 
 \begin{equation}\label{eq:5}
 \Phi_\lambda(v_\ve,z_\ve) \leq 1 + \ve^2 \lambda T \gamma + \eta_t^{\lambda}.
 \end{equation}
Therefore, 
  \begin{eqnarray}
& & P\{ \sup_{0 \leq t \leq T} \{ \|u_\ve - u\|^2 +  \|w_\ve - w\|^2\} \geq r^2\}  = P\{ \sup_{t \in [0,T]} \Phi_\lambda(v_\ve, z_\ve) \geq (1+\lambda r^2)^{\frac{1}{2}}\} \leq \\
\nonumber & & P\{ 1 + \ve^2 \lambda T \gamma + \sup_{t \in [0,T]} \eta_t^\lambda \geq (1 + \lambda r^2)^{\frac{1}{2}}\} = P\{e^{ \sup_{t \in [0,T]} \eta_t^\lambda} \geq e^{(1 + \lambda r^2)^{\frac{1}{2}} -  \ve^2 \lambda T \gamma - 1} \}. 
 \end{eqnarray}
 It follows from Ito's formula $e^{\eta_t^{\lambda}}$ is a martingale, which, in addition, satisfies
 $$\E e^{\eta_t^\lambda} = \E e^{\eta_0^\lambda} = 1.$$ 
Now, using Chebyshev inequality, Doob's martingale inequality, and the fact that $\sup_{t} e^{a(t)} = e^{\sup_t a(t)}$, we get
 \begin{equation}\label{eqn:6}
P\{ \sup_{0 \leq t \leq T} \{ \|u_\ve(t) - u(t)\|^2 +  \|w_\ve(t) - w(t)\|^2\} \geq r^2\} \leq e^{1+  (1 + \ve^2 \lambda T \gamma - \lambda r^2)^{\frac{1}{2}}}.
 \end{equation}
 It remains to optimize the latter expression in $\lambda$.  Setting
 \begin{equation*}\label{eqn:8}
\lambda  := \left(\frac{r}{2  \gamma \ve^2 T}\right)^2  - \frac{1}{r^2},
\end{equation*}
 the exponent can be estimated as 
 \[
 1+  (1 + \ve^2 \lambda T \gamma - \lambda r^2)^{\frac{1}{2}} \leq 1 - \frac{r^2}{4 \ve^2 \gamma T},
 \]
 yielding the desired result
 \[
P\{ \sup_{0 \leq t \leq T} \{\|u_\ve(t) - u(t)\|^2 +  \|w_\ve(t) - w(t)\|^2\} \geq r^2\} \leq 3 e^{- \frac{r^2}{4 \ve^2 \gamma T}}.
\] 
 \end{proof} 
We conclude this subsection with some examples of typical nonlinearities, satisfying the conditions of Theorem \ref{Thm: Large Deviations}.

\begin{example} Allen-Cahn nonlinearity
\begin{equation*}
f(u) := \eta(u^3-u).
\end{equation*}
The conditions \eqref{C2} and \eqref{monot} read as
\begin{align}
u f(u) &\geq a u^4 + b u^2 + K, \quad \mbox{ and} \label{D1_1}\\
(f(u_1)-f(u_2))(u_1-u_2) &\geq c_1 (u_1-u_2)^2.\label{D2_1}
\end{align}
Clearly,  condition \eqref{D1_1} is satisfied. In order to verify \eqref{D2_1}, note that
\[
(f(u_1)-f(u_2))(u_1-u_2) \geq \eta(u_1-u_2)^2 (u_1^2 + u_1 u_2 + u_2^2 - 1) \geq -\eta(u_1-u_2)^2.
\]
Thus the coefficient condition \eqref{coefficient condition} in \eqref{monot} holds, provided $\eta \leq \frac{\alpha}{C_p}$.
\end{example}

\begin{example} FitzHugh-Nagumo nonlinearity 
\[
\begin{cases}
f(u,w) = \eta[u(u-a)(u-1) + w], \\
g(u,w) = kw - du,
\end{cases} 
\]
with  $0<a<1$ and $k,d>0$.  Note that
\[
u \eta [u^3 - (a+1)u^2 + au] + \eta u w + kw^2 - d u w \geq \frac{\eta}{2} u^4 + a\eta u^2  - |\eta-d | (\frac{1}{2} w^2  + \frac{1}{2} u^2) + k w^2.
\]
Hence \eqref{C2} holds for $\frac{|\eta-d|}{2}  \leq k$. In order to verify the condition \eqref{coefficient condition} in \eqref{monot},  note that
\begin{align*}
&-\eta (-u_1^3 + u_2^3 + (1+a)(u_1^2 -u_2^2) -a (u_1-u_2))(u_1-u_2) + \\ & + (\eta-d) (w_1-w_2)(u_1-u_2) + k (w_1-w_2)^2 \geq - \eta \big((1+a)^2/3 -a\big) (u_1-u_2)^2 \\
&  -|\eta-d|\Big(\frac{(u_1-u_2)^2}{2} + \frac{(w_1-w_2)^2}{2}\Big) + k (w_1-w_2)^2 \\
&= -\Big((1+a)^2/3 -a)\eta + \frac{|\eta-d|}{2}\Big)(u_1-u_2)^2 + \Big(k-\frac{|\eta-d|}{2} \Big) (w_1-w_2)^2,
\end{align*}
where we used the elementary inequality
\[
-(u_1^2 + u_1 u_2 + u_2^2) + (1+a)(u_1+u_2) -a \leq \frac{1}{3}(1+a)^2 - a \in \left[\frac{1}{4}, \frac{1}{3}\right), 0<a<1.
\]
Consequently, \eqref{coefficient condition}  holds for $\eta$ and $d$ satisfying $\left(\frac{(1+a)^2}{3} -a\right)\eta + \frac{|\eta-d|}{2} \leq \frac{\alpha}{C_p}$ and $\frac{|\eta-d|}{2}  \leq k$. 
\end{example}

 \subsection{Convergence of stationary solutions}

\begin{equation}\label{BDESeps1}
\begin{cases}
  du_\ve  = [- \aA u_\ve - f(u_\ve ,w_\ve ) + I(x)] dt + \ve dW, \ t \geq 0,\\
  d w_\ve  =  -g(u_\ve ,w_\ve ) dt.
 \end{cases}
\end{equation} 
 
 Assume the conditions of Theorem 5.2 \cite{Mi} hold. Then for any $\ve \geq 0$ the equation \eqref{BDESeps1} has a unique stationary solution $z_\ve^*:=(u_\ve^*,w_\ve^*)$.
 
\begin{remark}
Note that $z_\ve$ may have different initial conditions for every $\ve$.
\end{remark} 
\begin{theorem}
Under the conditions of Theorem 5.2 \cite{Mi}, we have
 \[
 \sup_{t \in [0,T]} \E \|z^*_\ve(t) - z^*_0(t)\| \to 0, \ve \to 0.
 \]
 \end{theorem}
 
Following the lines of Theorem 6.3.2  of \cite{DapZab96}, we extend the Wiener process $W$ for negative $t$ by
\begin{equation}
\overline{W}(t) =
\begin{cases}
W(t), \text{ for } t \geq 0,\\
V(-t), \text{ if } t \leq 0,
\end{cases}
\end{equation}
and set $\overline{\mathcal{F}(t)} = \sigma\big(\overline{W}(s), s \leq t\big)$ for $t \in \R$, where $V(t),t \geq 0$ is another Wiener process, independent of $W(t)$.
For simplicity of our notation, we still denote the process $\overline{W(t)}$ by $W(t)$. For $n \geq 1$ denote $z_\ve(t,-n, z_0)$ to be the solution of \eqref{BDESeps1} with the initial condition $z_\ve(-n) = z_0 = (u_0, w_0) \in H \times H$. 
Using the energy estimate from Theorem 5.2 \cite{Mi}, we have
\begin{equation}\label{eqn2}
\frac{1}{2} \frac{d}{dt} \E \|z_\ve(t)\|^2  \leq \frac{1}{2} \ve^2 \gamma - \mu \E \|z_\ve(t)\|^2 + K_2,
\end{equation}
which, by Gronwall's inequality, implies the uniform in $\ve \in [0,1]$ bound
\[
\E \|z_\ve\|^2 \leq K_3(1 + \|z(0)\|^2_H).
\]
Furthermore, for $\ve_1 \geq 0$ and $\ve_2 \geq 0$,
\begin{equation}\label{BDESeps3}
\begin{cases}
  d(u_{\ve_1} -  u_{\ve_2}) = [- \aA (u_{\ve_1} -  u_{\ve_2}) - [f(u_{\ve_1} ,w_{\ve_1}) - f(u_{\ve_2} ,w_{\ve_2})] dt + (\ve_1 - \ve_2) dW, \ t \geq 0,\\
  d(w_{\ve_1} - w_{\ve_2})  =  -(g(u_{\ve_1} ,w_{\ve_1} ) - g(u_{\ve_2} ,w_{\ve_2})) dt.
 \end{cases}
\end{equation}  
Using the same reasoning as in \eqref{eqn2}, we have
 \[
 \frac{d}{dt} \left(\E \|z_{\ve_1} - z_{\ve_2}\|^2\right) \leq -\mu \E \|z_{\ve_1} - z_{\ve_2}\|^2 + (\ve_1 - \ve_2)^2 \gamma,
 \]
yielding
 \begin{multline}\label{eq4}
 \E  ( \|u_{\ve_1}(0,-n,u_0)-u_{\ve_2}(0,-n,u_0)\|^2 +  \|w_{\ve_1}(0,-n,u_0)-w_{\ve_2}(0,-n,u_0)\|^2) \\
  \leq (\ve_1-\ve_2)^2 \gamma.
 \end{multline}
Furthermore, as it was shown in the proof of Theorem 5.2 \cite{Mi}, we have
 \[
 \E \|u_{\ve_1}(0,-n, u_0)-u_{\ve_2}(0,-n, u_0)\|^2 \to \E \|u_{\ve_1}^{*}(0)-u_{\ve_2}^{*}(0)\|^2, \ n \to \infty
 \]
 and
 \[
 \E \|w_{\ve_1}(0,-n, u_0)-w_{\ve_2}(0,-n, u_0)\|^2 \to \E \|w_{\ve_1}^{*}(0)-w_{\ve_2}^{*}(0)\|^2, \ n \to \infty.
 \] 
 Passing to the limit as $n \to \infty$ in \eqref{eq4}, we obtain
 \begin{equation*}
 \E  ( \|u_{\ve_1}^{*}(0)-u_{\ve_2}^{*}(0)\|^2 +  \|w_{\ve_1}^{*}(0)-w_{\ve_2}^{*}(0)\|^2) \leq (\ve_1-\ve_2)^2 \gamma.
 \end{equation*}
 The bound above implies that the sequence $\{(u_{\ve}^{*}(0), w_{\ve}^{*}(0)\}$ is Cauchy, and  therefore
 \[
 \{(u_{\ve}^{*}(0), w_{\ve}^{*}(0)\} \to  \{(u_{0}^{*}(0), w_{0}^{*}(0)\}, \ \ \ve \to 0.
 \]
We are now in position to establish 
\[
 \E \sup_{t \in [0,T]} \|z_{\ve}^{*}(t) - z_{0}^{*}(t)\| \to 0, \ve \to 0. 
 \]
To this end, setting $v_\ve = u_\ve^* - u_0^*$ and $y_\ve = w_\ve^* - w_0^*$. Then 
\begin{align*}
& \|v_\ve(t)\|^2  + \|y_\ve(t)\|^2 = \|u_\ve^{*}(0) - u_0^{*}(0)\|^2 + \|w_\ve^*(0) - w_0^*(0)\|^2  + \\
& 2 \int_0^t [ \langle -\aA v_\ve, v_\ve \rangle  - (f(u_\ve^*,w_\ve^*) - f(u_0^*,w_0^*), u_\ve^* - u_0^*)  - \\
& - (g(u_\ve^*,w_\ve^*) - g(u_0^*,w_0^*), w_\ve^* - w_0^*)] ds + \ve^2 \gamma t + 2 \ve \int_0^t (v_\ve^*, d W_s).
\end{align*}
We may now proceed with estimating the stochastic analogously to \eqref{est_stochast}, and then use the Gronwall's inequality the same way as in \eqref{eqn2} to conclude that
\begin{multline*}
\E \sup_{t \in [0,T]}(\|u_\ve^*(t) - u_0^*(t)\|^2 + \|w_\ve^*(t) - w_0^*(t)\|^2) \\
\leq \E(\|u_\ve^*(0) - u_0^*(0)\|^2 + \|w_\ve^*(0) - w_0^*(0)\|^2) + C_1(T) \ve \to 0, \ \ve \to 0,
\end{multline*}
 which completes the proof. 
 
\section*{Acknowledgement} The research of Oleksandr Misiats was supported by Simons Collaboration
Grant for Mathematicians No. 854856. The work of Oleksiy Kapustyan and Oleksandr Stanzhytskyi was partially supported by the National Research Foundation of Ukraine No. F81/41743 and Ukrainian Government Scientific Research Grant No. 210BF38-01.

\bibliography{bibliographybid}

\end{document}